\documentclass[12pt,reqno]{amsart}
\usepackage{amssymb,amsmath,amsthm,hyperref}
\newcommand{\sg}{\textnormal{sg}}

\newtheorem{theorem}{Theorem}
\newtheorem{lemma}[theorem]{Lemma}
\newtheorem{corollary}[theorem]{Corollary}
\newtheorem{proposition}[theorem]{Proposition}

\theoremstyle{remark}

\theoremstyle{definition}

\numberwithin{theorem}{section} \numberwithin{equation}{section}
\numberwithin{example}{section}

\title{A double-sum Kronecker-type identity}

\author{Eric T. Mortenson}

\begin{document}

\date{10 August 2016}



\begin{abstract}
We prove a double-sum analog of an identity known to Kronecker and then express it in terms of functions studied by Appell and Kronecker's student Lerch, in so doing we show that the double-sum analog is of mixed mock modular form.  We also give related symmetric generaizations.
\end{abstract}

\address{Max-Planck-Institut f\"ur Mathematik, Vivatsgasse 7, 53111 Bonn, Germany}
\address{Universit\"at zu K\"oln, Mathematisches Institut, Weyertal 86-90, 50931 K\"oln, Germany}
\email{etmortenson@gmail.com}
\maketitle
\setcounter{section}{-1}

\section{Notation}

 Let $q$ be a nonzero complex number with $|q|<1$ and define $\mathbb{C}^*:=\mathbb{C}-\{0\}$.  Recall
\begin{gather}
(x)_n=(x;q)_n:=\prod_{i=0}^{n-1}(1-q^ix), \ \ (x)_{\infty}=(x;q)_{\infty}:=\prod_{i\ge 0}(1-q^ix),\notag \\
{\text{and }} \ \ j(x;q):=(x)_{\infty}(q/x)_{\infty}(q)_{\infty}=\sum_{n=-\infty}^{\infty}(-1)^nq^{\binom{n}{2}}x^n,\notag
\end{gather}
where in the last line the equivalence of product and sum follows from Jacobi's triple product identity.    Here $a$ and $m$ are integers with $m$ positive.  Define
\begin{gather*}
J_{a,m}:=j(q^a;q^m), \ \ J_m:=J_{m,3m}=\prod_{i\ge 1}(1-q^{mi}), \ {\text{and }}\overline{J}_{a,m}:=j(-q^a;q^m).
\end{gather*}
We will use the following definition of an Appell-Lerch function \cite{Ap, HM, L1, Zw}:
\begin{equation}
m(x,q,z):=\frac{1}{j(z;q)}\sum_{r=-\infty}^{\infty}\frac{(-1)^rq^{\binom{r}{2}}z^r}{1-q^{r-1}xz}.\label{equation:mdef-eq}
\end{equation}
\section{Introduction}

The following identity was known to Kronecker \cite{Kron1}, \cite[pp. 309-318]{Kron2}, see also A. Weil's monograph for Kronecker's proof \cite[pp. 70-71]{Weil}; however, Kronecker's identity is also a special case of Ramanujan's ${}_{1}\psi_{1}$-summation.  For  $x,y\in \mathbb{C}^*$ where $|q|<|x|<1$ and $y$ neither zero or an integral power of $q$
\begin{equation}
\sum_{r\in \mathbb{Z}}\frac{x^r}{1-yq^{r}}=\frac{(q)_{\infty}^2(xy,q/xy;q)_{\infty}}{(x,q/x,y,q/y;q)_{\infty}}.\label{equation:kronecker-original}
\end{equation}
\noindent If we place the additional restriction $|q|<|y|<1$, we have a more symmetric form,
\begin{equation}
\Big ( \sum_{r,s\ge 0}-\sum_{r,s<0} \Big)q^{rs}x^ry^s=\frac{J_1^3j(xy;q)}{j(x;q)j(y;q)}.\label{equation:kronecker}
\end{equation}
A natural question is what are the higher-dimensional generalizations of (\ref{equation:kronecker-original})?

In \cite{HM}, we expanded Hecke-type double sums in terms of Appell-Lerch functions and theta functions.  As an example, we showed for generic $x,y\in \mathbb{C}^*$,
\begin{align}
\Big( \sum_{r,s\ge 0}&- \sum_{r,s<0}\Big)(-1)^{r+s}x^ry^sq^{\binom{r}{2}+2rs+\binom{s}{2}}\label{equation:n1p1}\\
&=j(y;q)m\big (\frac{q^2x}{y^2},q^3,-1\big )+j(x;q)m\big (\frac{q^2y}{x^2},q^3,-1\big )
- \frac{yJ_{3}^3j(-x/y;q)j(q^2xy;q^3)}
{\overline{J}_{0,3}j(-qy^2/x;q^3)j(-qx^2/y;q^3)}.\notag
\end{align}

In \cite{M9}, we demonstrated how identity (\ref{equation:kronecker}) can by used to determine directly the theta-quotient term of Hecke-type doubles such as in (\ref{equation:n1p1}).  Indeed, one can actually see the right-hand side of (\ref{equation:kronecker}) within the extreme right-hand side of (\ref{equation:n1p1}).   In trying to determine the modularity of so-called Hecke-type triple-sums \cite{HL}, i.e. sums of the form
\begin{align}
\Big ( \sum_{r,s,t \ge 0}&+\sum_{r,s,t<0} \Big)(-1)^{r+s+t}x^ry^sz^tq^{a\binom{r}{2}+b\binom{s}{2}+c\binom{t}{2}+drs+ert+fst},
\end{align}
the natural question is to determine the modularity of 
\begin{align}
\Big ( \sum_{r,s,t \ge 0}&+\sum_{r,s,t<0} \Big)q^{rs+rt+st}x^ry^sz^t,
\end{align}
which would be a higher-dimensional generalization of (\ref{equation:kronecker}).  It turns out that there is a double-sum analog of (\ref{equation:kronecker-original}), it is our result, and it appears to be new.

\begin{theorem}\label{theorem:result} For  $x,y,z\in \mathbb{C}^*$ where $|q|<|y|<1$,  $|q|<|z|<1$, and $x$ neither zero or an integral power of $q$,
{\allowdisplaybreaks \begin{align}
\Big ( \sum_{s,t \ge 0}&-\sum_{s,t<0} \Big)\frac{q^{st}y^sz^t}{1-xq^{s+t}} \label{equation:thm-result}\\
&=  \frac{(yz,q^2/yz;q^2)_{\infty}}{(y,z,q/y,q/z;q)_{\infty}}
\frac{(q;q)_{\infty}^2}{(q^2;q^2)_{\infty}} \sum_{k\in \mathbb{Z}}\frac{(-1)^kq^{k^2}(yz)^k}{1+q^{2k}x}\notag \\
&\ \ \ \ \ + \frac{(xy,q^2/xy;q^2)_{\infty}}{(x,y,q/x,q/y;q)_{\infty}}
 \frac{(q;q)_{\infty}^2}{(q^2;q^2)_{\infty}}\sum_{k\in \mathbb{Z}} \frac{(-1)^kq^{k^2}(xy)^k}{1+q^{2k}z}\notag \\
&\ \ \ \ \ -2  \frac{(q^2;q^2)_{\infty}^3}{(x,y,z,q/x,q/y,q/z;q)_{\infty}} \frac{(xy,xz,yz,q^2/xy,q^2/xz,q^2/yz;q^2)_{\infty}}{(-x,-y,-z,-q^2/x,-q^2/y,-q^2/z;q^2)_{\infty}}.\notag \\
&\ \ \ \ \ + \frac{(xz,q^2/xz;q^2)_{\infty}}{(x,z,q/x,q/z;q)_{\infty}}
 \frac{(q;q)_{\infty}^2}{(q^2;q^2)_{\infty}}  \sum_{k\in \mathbb{Z}} \frac{(-1)^kq^{k^2}(xz)^k}{1+q^{2k}y}.\notag 
\end{align}}%
\end{theorem}

\noindent Restricting $x$ and using the Appell-Lerch function notation, we have the symmetric

\begin{corollary} For  $x,y,z\in \mathbb{C}^*$ where $|q|<|x|<1$, $|q|<|y|<1$, and $|q|<|z|<1$,
\begin{align}
\Big ( \sum_{r,s,t \ge 0}&+\sum_{r,s,t<0} \Big)q^{rs+rt+st}x^ry^sz^t\\
&=  \frac{J_1^3j(yz;q)}{j(y;q)j(z;q)} m \big (-\frac{qx}{yz},q^{2},qyz \big )+ \frac{J_1^3j(xy;q)}{j(x;q)j(y;q)} m \big (-\frac{qz}{xy},q^{2},qxy \big )\notag \\
&\ \ \ \ \ -2 \frac{J_1^3J_2^3}{j(x;q)j(y;q)j(z;q)}\frac{j(xy;q^2)j(xz;q^2)j(yz;q^2)}{j(-x;q^2)j(-y;q^2)j(-z;q^2)}\notag\\
&\ \ \ \ \ + \frac{J_1^3j(xz;q)}{j(x;q)j(z;q)} m \big (-\frac{qy}{xz},q^{2}, qxz\big ).\notag 
\end{align}
\end{corollary}

\noindent In particular, where Kronecker's (\ref{equation:kronecker-original}) is modular, we see that our new Theorem \ref{theorem:result} is in fact mixed mock modular \cite{Za}.  

One could also ask whether or not there are analogous higher-dimensional generalizations of Hickerson's \cite[$(1.30)$, $(1.32)$]{H1}.  If we introduced the more compact notation
\begin{equation*}
\sg(r):=
\begin{cases}
1 \textup{ if }r\ge0,\\
-1 \textup{ if }r<0,
\end{cases}
\end{equation*}
where $r$ is an integer, Hickerson's two identities \cite[$(1.30)$, $(1.32)$]{H1} read respectively
\begin{gather}
\sum_{\substack{ \sg(r)=\sg(s)\\r\equiv s \pmod 2}}\sg(r)q^{rs}x^ry^s
=\frac{J_{2,4}j(qxy;q^2)j(-qxy^{-1};q^2)j(x^2y^2;q^4)}{j(x^2;q^2)j(y^2;q^2)},\label{equation:HThm1.6}\\
\sum_{\substack{ \sg(r)=\sg(s)\\r\not \equiv s \pmod 2}}\sg(r)q^{rs}x^ry^s
=\frac{yJ_{2,4}j(xy;q^2)j(-xy^{-1};q^2)j(q^2x^2y^2;q^4)}{j(x^2;q^2)j(y^2;q^2)},\label{equation:Hid1pt32}
\end{gather}
where for both identities one has the restrictions $|q|<|x|<1$ and $|q|<|y|<1$.  

It turn out, that when $r$, $s$, and $t$ are required to have the same parity, we have
\begin{theorem} \label{theorem:same-parity} For  $x,y,z\in \mathbb{C}^*$ where $|q|<|x|<1$, $|q|<|y|<1$, and $|q|<|z|<1$,
\begin{align}
&\sum_{\substack{ \sg(r)=\sg(s)=\sg(t)\\r\equiv s \equiv t \pmod 2}}q^{rs+rt+st}x^ry^sz^t\label{equation:parity-same}\\
&\ \ \ \ \ \ \ \ \ \ = \frac{J_4^3j(y^2z^2;q^4)}{j(y^2;q^4)j(z^2;q^4)}  m \big (-\frac{qx}{yz},q^{2},-yz \big ) 
+\frac{J_4^3j(x^2z^2;q^4)}{j(x^2;q^4)j(z^2;q^4)}  m \big (-\frac{qy}{xz},q^{2},-xz \big ) \notag\\
&\ \ \ \ \ \ \ \ \ \ \ \ \ \ \ +\frac{J_1^3J_2^3}{j(x;q)j(y;q)j(z;q)} \frac{j(xy;q^2)j(xz;q^2)j(yz;q^2)}{j(-x;q^2)j(-y;q^2)j(-z;q^2)}\notag\\
&\ \ \ \ \ \ \ \ \ \ \ \ \ \ \  + \frac{J_4^3j(x^2y^2;q^4)}{j(x^2;q^4)j(y^2;q^4)}  m \big (-\frac{qz}{xy},q^{2},-xy \big ).  \notag
\end{align}
\end{theorem}
\noindent When $r$, $s$, and $t$ do not all have the same parity, we have for example
\begin{theorem} \label{theorem:different-parity} For  $x,y,z\in \mathbb{C}^*$ where $|q|<|x|<1$, $|q|<|y|<1$, and $|q|<|z|<1$,
\begin{align}
&\sum_{\substack{ \sg(r)=\sg(s)=\sg(t)\\r\equiv s \not \equiv t \pmod 2}}q^{rs+rt+st}x^ry^sz^t\label{equation:parity-dif}\\
&= z \frac{J_4^3j(q^2y^2z^2;q^4)}{j(q^2y^2;q^4)j(z^2;q^4)}  m \big (-\frac{qx}{yz},q^{2},-qyz \big )  \notag
+ z \frac{J_4^3j(q^2x^2z^2;q^4)}{j(q^2x^2;q^4)j(z^2;q^4)} m \big (-\frac{qy}{xz},q^{2}, -qxz\big ) \notag\\
&\ \ \ \ \ -z \frac{J_1^3J_2^3}{j(x;q)j(y;q)j(z;q)} \frac{j(xy;q^2)j(qxz;q^2)j(qyz;q^2)}{j(-xq;q^2)j(-yq;q^2)j(-z;q^2)}\notag\\
&\ \ \ \ \  
 - \frac{q}{xy} \frac{J_4^3j(x^2y^2;q^4)}{j(q^2x^2;q^4)j(q^2y^2;q^4)}  m \big (-\frac{qz}{xy},q^{2},-xy \big ).\notag
\end{align}
\end{theorem}

In Section \ref{section:notation}, we recall useful facts on theta functions and Appell-Lerch functions.  In Section \ref{section:functional}, we demonstrate that the left and right-hand sides of Theorem \ref{theorem:result} satisfy the same functional equation.  In Section \ref{section:analytic}, we show that the difference between the left and right-hand sides of Theorem \ref{theorem:result} is analytic for $x\ne0$.  This we call the difference function.  In Section \ref{section:proof}, we prove Theorem \ref{theorem:result} by expressing the difference function in terms of a Laurent series and then showing how the functional equations of Section \ref{section:functional} force the difference function to be zero.  In Sections \ref{section:same} and \ref{section:different}, we prove Theorems \ref{theorem:same-parity} and \ref{theorem:different-parity}, respectively.  In Section \ref{section:starting}, we sketch how one can guess the right-hand side of our three new theorems up to a theta function.  We also point out ideas for alternate proofs.

\section*{Acknowledgements}
We would like to thank Christian Krattenthaler, Wadim Zudilin, Ole Warnaar, and the two referees for helpful comments and suggestions.  In particular, we would like to thank the second referee who suggested finding generalizations of (\ref{equation:HThm1.6}) and (\ref{equation:Hid1pt32}).

\section{Preliminaries}\label{section:notation}

We have the general identities:
\begin{subequations}
{\allowdisplaybreaks \begin{gather}
j(q^n x;q)=(-1)^nq^{-\binom{n}{2}}x^{-n}j(x;q), \ \ n\in\mathbb{Z},\label{equation:j-elliptic}\\
j(x;q)=j(q/x;q)=-xj(x^{-1};q)\label{equation:j-inv},\\
j(x;q)={J_1}j(x;q^2)j(xq;q^2)/{J_2^2}, \label{equation:j-mod-2}\\
j(x^2;q^2)={J_2}j(x;q)j(-x;q)/{J_1^1}. \label{equation:j-mod-dec}
\end{gather}}%
\end{subequations}

In addition, the following proposition will be useful in computing residues. 
\begin{proposition} \cite[Theorem 1.3]{H1} \label{proposition:H1Thm1.3} Define $G(z):={1}/{j(\beta z^b;q^m)}.$  $G(z)$ is meromorphic for $z\ne 0$ with simple poles at points $z_0$ such that $z_0^b=q^{km}/\beta$.  The residue at such $z_0$ is ${(-1)^{k+1}q^{m\binom{k}{2}}z_0}/{bJ_m^3}$.
\end{proposition}

The Appell-Lerch function $m(x,q,z)$ satisfies several well-known functional equations and identities, which we collect in the form of a proposition, 

\begin{proposition}  For generic $x,z\in \mathbb{C}^*$
{\allowdisplaybreaks \begin{subequations}
\begin{gather}
m(x,q,z)=m(x,q,qz),\label{equation:mxqz-fnq-z}\\
m(x,q,z)=x^{-1}m(x^{-1},q,z^{-1}),\label{equation:mxqz-flip}\\
m(qx,q,z)=1-xm(x,q,z),\label{equation:mxqz-fnq-x}\\
m(x,q,z)=m(x,q,x^{-1}z^{-1})\label{equation:mxqz-fnq-newz}.
\end{gather}
\end{subequations}}%
\end{proposition}

Rewriting (\ref{equation:mxqz-fnq-x}), we have
\begin{equation}
m(x,q,z)=1-q^{-1}xm(q^{-1}x,q,z)\label{equation:mxqz-altdef0intro}.
\end{equation}
In \cite[Section $3$]{HM} we introduced a heuristic which guided our further study of the Appell-Lerch function $m(x,q,z)$ and Hecke-type double-sums.  If we iterate (\ref{equation:mxqz-altdef0intro}), we obtain
\begin{equation}
m(x,q,z)\sim \sum_{r\ge 0}(-1)^rq^{-\binom{r+1}{2}}x^r.\label{equation:mxqz-heuristic}
\end{equation}
Of course, we cannot use an equal sign here, since the infinite series on the right diverges for $|q|<1$.  However, it is often useful to think of $m(x,q,z)$ as a partial theta series with $q$ replaced by $q^{-1}$.  Roughly speaking, we may think of ``$\sim$'' as congruence `mod theta'.  For example, since the series (\ref{equation:mxqz-heuristic}) does not depend on $z$, we may write
\begin{equation}
m(x,q,z_0)\sim m(x,q,z_1),
\end{equation}
where $z_0$ and $z_1$ are generic placeholders.  In fact, the difference between these two quantities is a theta function, as we see in the following well-known result,

\begin{proposition} \label{proposition:changing-z}For generic $x,z_0,z_1\in \mathbb{C}^*$
\begin{equation}
m(x,q,z_1)-m(x,q,z_0)=\frac{z_0J_1^3j(z_1/z_0;q)j(xz_0z_1;q)}{j(z_0;q)j(z_1;q)j(xz_0;q)j(xz_1;q)}.
\end{equation}
\end{proposition}
\noindent A specialisation of Propositon \ref{proposition:changing-z} that we will use later reads
{\allowdisplaybreaks \begin{subequations}
\begin{gather}
m(-\frac{z}{qy},q^2,qy) -m(-\frac{z}{qy},q^2,-y)
=\frac{yq}{z}\frac{J_2^3j(-q;q^2)j(yz;q^2)}{j(qy;q^2)j(-y;q^2)j(-z;q^2)j(qz;q^2)}.\label{equation:change-n-even}
\end{gather}
\end{subequations}}%

With $R$ and $S$ denoting integers, we recall the useful \cite[$(1.15)$]{H2}
\begin{equation}
\sum_{\sg(r)=\sg(s)}\sg(r)c_{r,s}=\sum_{\sg(r)=\sg(s)}\sg(r)c_{r+R,s+S}+\sum_{r=0}^{R-1}\sum_{s}c_{r,s}+\sum_{s=0}^{S-1}\sum_{s}c_{r,s},
\label{equation:H2id1.15}
\end{equation}
and we point out that we will be observing the convention \cite{H2}:  for $b<a$, 
\begin{equation}
\sum_{r=a}^{b}c_r:=-\sum_{r=b+1}^{a-1}c_r.
\end{equation}
To prevent later confusion for the unfamiliar reader, we also point out examples such as
\begin{equation}
\sum_{r=1}^{-1}c_r=-\sum_{r=0}^{0}c_r  \textup{ and } \sum_{r=1}^{0}c_r=-\sum_{r=1}^{0}c_r=0.
\end{equation}

Our convention allows us to combine two seemingly different cases into one case.  For example, induction arguments generalizing (\ref{equation:mxqz-fnq-x}) and (\ref{equation:mxqz-altdef0intro}) yield two different results.  However, our summation convention allows us to combine the two results into one:
\begin{equation}
m(q^nx,q,z)=\sum_{k=0}^{n-1}(-1)^kq^{k(n-1)-\binom{k}{2}}x^k+(-1)^nq^{\binom{n}{2}}x^nm(x,q,z), \ \textup{for $n\in\mathbb{Z}$}.\label{equation:mxqz-induction}
\end{equation}

\section{Two Functional equations}\label{section:functional}
We define
\begin{equation}
F(x,y,z;q):=\Big ( \sum_{s,t \ge 0}-\sum_{s,t<0} \Big)\frac{q^{st}y^sz^t}{1-xq^{s+t}} \label{equation:big-F-def}
\end{equation}
and
{\allowdisplaybreaks \begin{align}
G(x,y,z;q):&=\frac{J_1^3j(yz;q)}{j(y;q)j(z;q)} m \big (-\frac{qx}{yz},q^{2},qyz \big )
+ \frac{J_1^3j(xy;q)}{j(x;q)j(y;q)} m \big (-\frac{qz}{xy},q^{2},qxy\big ) \notag \\
&\ \ \ \ \ + \frac{J_1^3j(xz;q)}{j(x;q)j(z;q)} m \big (-\frac{qy}{xz},q^{2}, qxy\big )\notag \\
&\ \ \ \ \ \ \ \ \ \ -2 \frac{J_1^3J_2^3}{j(x;q)j(y;q)j(z;q)} \frac{j(xy;q^2)j(xz;q^2)j(yz;q^2)}{j(-x;q^2)j(-y;q^2)j(-z;q^2)}.\label{equation:big-G-def}
\end{align}}%
\begin{proposition}  \label{proposition:prop-functional} For  $x,y,z\in \mathbb{C}^*$ where $|q|<|y|<1$, $|q|<|z|<1$, and $x$ is generic, the functions $F(x,y,z;q)$ and $G(x,y,z;q)$ satisfy the functional equation
{\allowdisplaybreaks \begin{align}
M(q^2x,y,z;q)&=\frac{xq}{yz}M(x,y,z;q) +\frac{J_1^3j(yz;q)}{j(y;q)j(z;q)} \\
&\ \ \ \ \ -\frac{xq}{yz} \frac{J_1^3j(xy;q)}{j(x;q)j(y;q)} -\frac{xq}{yz} \frac{J_1^3j(xz;q)}{j(x;q)j(z;q)}.\notag
\end{align}}%
\end{proposition}
\begin{proof} 
Using Kronecker's identity (\ref{equation:kronecker}), we obtain
{\allowdisplaybreaks \begin{align*}
\frac{J_1^3j(yz;q)}{j(y;q)j(z;q)}&=\Big ( \sum_{s,t \ge 0}-\sum_{s,t<0} \Big)\frac{q^{st}y^sz^t}{1-xq^{s+t+2}}(1-xq^{s+t+2})\\
&=\Big ( \sum_{s,t \ge 0}-\sum_{s,t<0} \Big)\frac{q^{st}y^sz^t}{1-q^2xq^{s+t}}-x\Big ( \sum_{s,t \ge 0}-\sum_{s,t<0} \Big)\frac{q^{st+s+t+2}y^sz^t}{1-xq^{s+t+2}}\\
&=F(q^2x,y,z;q)-\frac{xq}{yz}\Big ( \sum_{s,t \ge 0}-\sum_{s,t<0} \Big)\frac{q^{(1+s)(1+t)}y^{s+1}z^{1+t}}{1-xq^{s+1+t+1}}\\
&=F(q^2x,y,z;q)\\
&\ \ \ \ \ -\frac{xq}{yz}\Big [ \Big ( \sum_{s,t \ge 0} -\sum_{s,t<0} \Big)\frac{q^{st}y^{s}z^{t}}{1-xq^{s+t}}-\sum_{t\in\mathbb{Z}}\frac{z^{t}}{1-xq^{t}}-\sum_{s\in\mathbb{Z}}\frac{y^{s}}{1-xq^{s}}\Big ]\\
&=F(q^2x,y,z;q)
-\frac{xq}{yz} F(x,y,z;q)+\frac{xq}{yz}\frac{J_1^3j(xz;q)}{j(x;q)j(z;q)}+\frac{xq}{yz}\frac{J_1^3j(xy;q)}{j(x;q)j(y;q)},
\end{align*}}%
where in the last two equalities we have used (\ref{equation:H2id1.15}) and (\ref{equation:kronecker-original}).  For $G(x,y,z;q)$, 
{\allowdisplaybreaks \begin{align*}
G(&q^2x,y,z;q)\\
&=\frac{J_1^3j(yz;q)}{j(y;q)j(z;q)} m \big (-\frac{q^3x}{yz},q^{2},qyz \big )
+ \frac{J_1^3j(q^2xy;q)}{j(q^2x;q)j(y;q)} m \big (-\frac{qz}{q^2xy},q^{2},q^3xy\big )\\
&\ \ \ \ \ + \frac{J_1^3j(q^2xz;q)}{j(q^2x;q)j(z;q)} m \big (-\frac{qy}{q^2xz},q^{2}, q^3xy\big )\\
&\ \ \ \ \ \ \ \ \ \ -2 \frac{J_1^3J_2^3}{j(q^2x;q)j(y;q)j(z;q)} \frac{j(q^2xy;q^2)j(q^2xz;q^2)j(yz;q^2)}{j(-q^2x;q^2)j(-y;q^2)j(-z;q^2)}\\
&=\frac{xq}{yz}G(x,y,z;q) +\frac{J_1^3j(yz;q)}{j(y;q)j(z;q)} -\frac{xq}{yz} \frac{J_1^3j(xy;q)}{j(x;q)j(y;q)} -\frac{xq}{yz} \frac{J_1^3j(xz;q)}{j(x;q)j(z;q)},
\end{align*}}%
where we have used (\ref{equation:mxqz-fnq-z}), (\ref{equation:mxqz-fnq-x}), and (\ref{equation:j-elliptic}).
\end{proof}

\section{An Analytic Function}\label{section:analytic}
The goal of this section is to establish the following proposition.
\begin{proposition}  \label{proposition:H-analytic} We fix  $y,z\in \mathbb{C}^*$ where $|q|<|y|<1$ and $|q|<|z|<1$.  The function
\begin{align*}
H(x,y,z;q):=F(x,y,z;q)-G(x,y,z;q),
\end{align*}
where $F(x,y,z;q)$ (resp. $G(x,y,z;q)$) is as defined in (\ref{equation:big-F-def}) (resp. (\ref{equation:big-G-def})), is analytic for $x\ne 0$.
\end{proposition}

Let us decompose our function $G(x,y,z;q)$ as
\begin{equation}
G(x,y,z;q)=G_1(x,y,z;q)+G_2(x,y,z;q),
\end{equation}
where
{\allowdisplaybreaks \begin{align}
G_1&(x,y,z;q):= \frac{j(yz;q^2)}{j(y;q)j(z;q)} \frac{J_1^4}{J_2^2}\sum_{k\in \mathbb{Z}}\frac{(-1)^kq^{k^2}(yz)^k}{1+q^{2k}x} \label{equation:G1-def}\\
&+ \frac{j(xy;q^2)}{j(x;q)j(y;q)}  \frac{J_1^4}{J_2^2}\sum_{k\in \mathbb{Z}}\frac{(-1)^kq^{k^2}(xy)^k}{1+q^{2k}z}  + \frac{j(xz;q^2)}{j(x;q)j(z;q)} \frac{J_1^4}{J_2^2}\sum_{k\in \mathbb{Z}}\frac{(-1)^kq^{k^2}(xz)^k}{1+q^{2k}y},\notag
\end{align}}%
and
\begin{align}
G_2(x,y,z;q):= -2 \frac{J_1^3J_2^3}{j(x;q)j(y;q)j(z;q)} \frac{j(xy;q^2)j(xz;q^2)j(yz;q^2)}{j(-x;q^2)j(-y;q^2)j(-z;q^2)}.\label{equation:G2-def}
\end{align}

With our notation in place, we now proceed with a series of lemmas and conclude the section with the proof of Proposition \ref{proposition:H-analytic}.

\begin{lemma} \label{lemma:F1-n} For fixed $y,z\in\mathbb{C}^*$ where $|q|<|y|<1$ and $|q|<|z|<1$, the function $F(x,y,z;q)$ has simple poles at $x_0=q^{n}$, where $n\in\mathbb{Z}$, with respective residues
\begin{equation}
\sum_{s=1}^{n-1}\frac{q^{s(n-s)+n}}{y^{s}z^{n-s}}.
\end{equation}
\end{lemma}
\begin{proof} Without loss of generality, we assume that $n\ge 1$.  We then begin with
{\allowdisplaybreaks \begin{align*}
\lim_{x\rightarrow q^{n}}(x-q^{n})\Big ( -\sum_{\substack{s+t=n\\ s,t\ge 1}}\frac{q^{s(n-s)}y^{-s}z^{-(n-s)}}{1-xq^{-n}}\Big )
&=\lim_{x\rightarrow q^{n}}(x-q^{n})\sum_{\substack{s+t=n\\ s,t\ge 1}}\frac{q^{s(n-s)+n}y^{-s}z^{-(n-s)}}{x-q^{n}}\\
&=\sum_{\substack{s+t=n\\ s,t\ge 1}}\frac{q^{s(n-s)+n}}{y^{s}z^{n-s}},
\end{align*}}%
and the result follows.
\end{proof}

\begin{lemma} \label{lemma:G1-n} For fixed $y,z\in\mathbb{C}^*$ where $|q|<|y|<1$ and $|q|<|z|<1$, the function $G_1(x,y,z;q)$ has simple poles at $x_0=q^{n}$, where $n\in\mathbb{Z}$, with respective residues
\begin{equation}
\sum_{s=1}^{n-1}\frac{q^{s(n-s)+n}}{y^{s}z^{n-s}}-\frac{q^{m^2+2m}}{y^mz^m}\cdot \frac{J_2^3j(-q;q^2)j(yz;q^2)}{j(qy;q^2)j(-y;q^2)j(-z;q^2)j(qz;q^2)}
\end{equation}
for $n=2m$, and
\begin{equation}
\sum_{s=1}^{n-1}\frac{q^{s(n-s)+n}}{y^{s}z^{n-s}} 
 +\frac{q^{m^2+3m+1}}{y^{m}z^{m}}\cdot \frac{J_2^3j(-1;q^2)j(yz;q^2)}{j(y;q^2)j(-y;q^2)j(z;q^2)j(-z;q^2)}
\end{equation}
for $n=2m+1$.
\end{lemma}
\begin{proof} Beginning with Proposition \ref{proposition:H1Thm1.3}, we have
{\allowdisplaybreaks \begin{align*}
&\lim_{x\rightarrow q^n}(x-q^n) G_1(x,y,z;q)\\
&=(-1)^{n+1}q^{\binom{n+1}{2}}\frac{j(q^ny;q^2)}{j(y;q)}\frac{J_1}{J_2^2}\sum_{k\in\mathbb{Z}}\frac{(-1)^kq^{k^2-k}(q^{n+1}y)^k}{1+q^{2k}z}\\
&\ \ \ \ \ + (-1)^{n+1}q^{\binom{n+1}{2}}\frac{j(q^nz;q^2)}{j(z;q)}\frac{J_1}{J_2^2}\sum_{k\in\mathbb{Z}}\frac{(-1)^kq^{k^2-k}(q^{n+1}z)^k}{1+q^{2k}y}\\
&=(-1)^{n+1}q^{\binom{n+1}{2}}\frac{j(q^ny;q^2)}{j(y;q)}\frac{J_1}{J_2^2}j(q^{n+1}y;q^2)m(-\frac{q^2z}{q^{n+1}y},q^2,q^{n+1}y)&(\textup{by }(\ref{equation:mdef-eq}))\\
&\ \ \ \ \ + (-1)^{n+1}q^{\binom{n+1}{2}}\frac{j(q^nz;q^2)}{j(z;q)}\frac{J_1}{J_2^2}j(q^{n+1}z;q^2)m(-\frac{q^2y}{q^{n+1}z},q^2,q^{n+1}z)\\
&=-q^{n}y^{-n}m(-\frac{q^2z}{q^{n+1}y},q^2,q^{n+1}y) -q^{n}z^{-n}m(-\frac{q^2y}{q^{n+1}z},q^2,q^{n+1}z)&(\textup{by }(\ref{equation:j-mod-2}), (\ref{equation:j-elliptic}))\\
&=\frac{q^{2n-1}y}{zy^n}m(-\frac{q^{n-1}y}{z},q^2,-z) +\frac{q^{2n-1}z}{yz^n}m(-\frac{q^{n-1}z}{y},q^2,-y).&(\textup{by }(\ref{equation:mxqz-flip}), (\ref{equation:mxqz-fnq-newz}))
\end{align*}}%
If $n=2m$, then
{\allowdisplaybreaks \begin{align*}
&\frac{q^{2n-1}y}{zy^n}m(-\frac{q^{n-1}y}{z},q^2,-z) +\frac{q^{2n-1}z}{yz^n}m(-\frac{q^{n-1}z}{y},q^2,-y)\\
&=\frac{q^{4m}y}{qzy^{2m}}m(-\frac{q^{2m}y}{qz},q^2,-z) +\frac{q^{4m}z}{qyz^{2m}}m(-\frac{q^{2m}z}{qy},q^2,-y)\\
&=\frac{q^{4m}y}{qzy^{2m}}\Big ( \sum_{k=0}^{m-1}q^{2k(m-1)-2\binom{k}{2}}(\frac{y}{qz})^k+q^{2\binom{m}{2}}(\frac{y}{qz})^{m}m(-\frac{y}{qz},q^2,-z) \Big )&(\textup{by }(\ref{equation:mxqz-induction}))\\
&\ \ \ \ \ +\frac{q^{4m}z}{qyz^{2m}}\Big ( \sum_{k=0}^{m-1}q^{2k(m-1)-2\binom{k}{2}}(\frac{z}{qy})^k+q^{2\binom{m}{2}}(\frac{z}{qy})^{m}m(-\frac{z}{qy},q^2,-y) \Big )\\
&=  \sum_{k=0}^{m-1}\frac{q^{2m(k+1)+2m-(k+1)^2}}{z^{k+1}y^{2m-(k+1)}}+ \frac{y}{qz}\frac{q^{m^2+2m}}{y^mz^m}m(-\frac{y}{qz},q^2,-z) \\
&\ \ \ \ \ + \sum_{k=0}^{m-1}\frac{q^{2m(k+1)+2m-(k+1)^2}}{y^{k+1}z^{2m-(k+1)}}+ \frac{z}{qy}\frac{q^{m^2+2m}}{y^mz^m}m(-\frac{z}{qy},q^2,-y)\\
&=  \sum_{k=1}^{m}\frac{q^{2mk+2m-k^2}}{z^{k}y^{2m-k}}- \frac{q^{m^2+2m}}{y^mz^m}-\frac{z}{qy} \frac{q^{m^2+2m}}{y^mz^m}m(-\frac{z}{qy},q^2,-\frac{1}{z})&(\textup{by }(\ref{equation:mxqz-flip}),\ (\ref{equation:mxqz-fnq-x})) \\
&\ \ \ \ \ + \sum_{k=1}^{m}\frac{q^{2mk+2m-k^2}}{y^{k}z^{2m-k}}+ \frac{z}{qy}\frac{q^{m^2+2m}}{y^mz^m}m(-\frac{z}{qy},q^2,-y)\\
&=\sum_{s=1}^{n-1}\frac{q^{s(n-s)+n}}{y^{s}z^{n-s}}-\frac{z}{qy}\frac{q^{m^2+2m}}{y^mz^m}\Big [ m(-\frac{z}{qy},q^2,qy)
 -m(-\frac{z}{qy},q^2,-y)\Big ].&( \textup{by }(\ref{equation:mxqz-fnq-newz}))
\end{align*}}%
The result then follows from (\ref{equation:change-n-even}).  The case $n=2m+1$ is similar and will be omitted.
\end{proof}

\begin{lemma} \label{lemma:G2-n} For fixed $y,z\in\mathbb{C}^*$ where $|q|<|y|<1$ and $|q|<|z|<1$, the function $G_2(x,y,z;q)$ has simple poles at $x_0=q^{n}$, where $n\in\mathbb{Z}$, with respective residues
\begin{equation}
\frac{q^{m^2+2m}}{y^mz^m}\cdot \frac{J_2^3j(-q;q^2)j(yz;q^2)}{j(-y;q^2)j(-z;q^2)j(yq;q^2)j(zq;q^2)}
\end{equation}
for $n=2m$, and
\begin{equation}
-\frac{q^{m^2+3m+1}}{y^mz^m}\cdot \frac{J_2^3j(-1;q^2)j(yz;q^2)}{j(y;q^2)j(-y;q^2)j(z;q^2)j(-z;q^2)}
\end{equation}
for $n=2m+1$.
\end{lemma}
\begin{proof}  Let us consider the case $n=2m$.  Using Proposition \ref{proposition:H1Thm1.3}, we have
{\allowdisplaybreaks \begin{align*}
\lim_{x\rightarrow q^n}&(x-q^n) G_2(x,y,z;q)\\
&=(-1)^{2m}q^{\binom{2m+1}{2}}\frac{2J_1^3J_2^3}{J_1^3j(y;q)j(z;q)} \frac{j(q^{2m}y;q^2)j(q^{2m}z;q^2)j(yz;q^2)}{j(-q^{2m};q^2)j(-y;q^2)j(-z;q^2)}\\
&=\frac{q^{m^2+2m}}{y^mz^m}\cdot  \frac{2J_2^3}{j(y;q)j(z;q)} \frac{j(y;q^2)j(z;q^2)j(yz;q^2)}{j(-1;q^2)j(-y;q^2)j(-z;q^2)}&(\textup{by }(\ref{equation:j-elliptic}))\\
&=\frac{q^{m^2+2m}}{y^mz^m}\cdot \frac{2J_2^3J_2^4}{J_1^2j(yq;q^2)j(zq;q^2)} \frac{j(yz;q^2)}{j(-1;q^2)j(-y;q^2)j(-z;q^2)}&(\textup{by }(\ref{equation:j-mod-2})).
\end{align*}}%
The result then follows from the product rearrangements $\overline{J}_{0,2}=2J_4^2/J_2$ and $\overline{J}_{1,2}=J_2^5/J_1^2J_4^2$.  The case $n=2m+1$  is similar and will be omitted.
\end{proof}

\begin{lemma} \label{lemma:G1-minus2n} For fixed $y,z\in\mathbb{C}^*$ where $|q|<|y|<1$ and $|q|<|z|<1$, the function $G_1(x,y,z;q)$  has simple poles at $x_0=-q^{2n}$, where $n\in\mathbb{Z}$, with respective residues
\begin{equation}
(-1)^n\frac{q^{n^2+2n}}{y^nz^n} \cdot \frac{j(yz;q^2)}{j(y;q)j(z;q)}\cdot \frac{J_1^4}{J_2^2}.
\end{equation}
\end{lemma}
\begin{proof} This is immediate from the definition of $G_1(x,y,z;q)$, e.g. (\ref{equation:G1-def}).
\end{proof}

\begin{lemma} \label{lemma:G2-minus2n} For fixed $y,z\in\mathbb{C}^{*}$ where $|q|<|y|<1$ and $|q|<|z|<1$, the function $G_2(x,y,z;q)$ has simple poles at $x_0=-q^{2n}$, where $n\in\mathbb{Z}$, with respective residues
\begin{equation}
-(-1)^n\frac{q^{n^2+2n}}{y^nz^n} \cdot \frac{j(yz;q^2)}{j(y;q)j(z;q)}\cdot \frac{J_1^4}{J_2^2}.
\end{equation}
\end{lemma}
\begin{proof} Using Proposition \ref{proposition:H1Thm1.3}, we have
{\allowdisplaybreaks \begin{align*}
&\lim_{x\rightarrow x_0}(x-x_0)G_2(x,y,z;q)\\
&=-2\cdot \frac{J_1^3J_2^3}{j(-q^{2n};q)j(y;q)j(z;q)}\cdot \frac{-(-1)^{n+1}q^{2\binom{n}{2}}q^{2n}}{J_2^3}\cdot
\frac{j(-q^{2n}y;q^2)j(-q^{2n}z;q^2)j(yz;q^2)}{j(-y;q^2)j(-z;q^2)}\\
&=- (-1)^{n}\frac{q^{n^2+2n}}{y^nz^n}\cdot \frac{j(yz;q^2)}{j(y;q)j(z;q)}\cdot \frac{J_1^4}{J_2^2},
\end{align*}}%
where we have used (\ref{equation:j-elliptic}) and the product rearrangement $\overline{J}_{0,1}=2J_2^2/J_1$.
\end{proof}

\begin{proof}[Proof of Proposition \ref{proposition:H-analytic}]  For a fixed $y$ and $z$ as in the proposition, we note that the only potential singularities of $H(x,y,z;q)$ are simple poles at $x_0=q^n$ and $x_0=-q^{2n}$ where $n\in\mathbb{Z}$.  Simple poles of the form $x_0=q^n$ occur in the functions $F(x,y,z;q)$, $G_1(x,y,z;q)$, and $G_2(x,y,z;q)$, so by Lemmas \ref{lemma:F1-n}, \ref{lemma:G1-n}, and \ref{lemma:G2-n}, we know that the residues sum to zero.  Simple poles of the form $x_0=-q^{2n}$ occur in the functions $G_1(x,y,z;q)$ and $G_2(x,y,z;q)$, so by Lemmas \ref{lemma:G1-minus2n} and \ref{lemma:G2-minus2n}, we know that the residues sum to zero.  Hence the function $H(x,y,z;q)$ is analytic for $x\ne 0$.
\end{proof}

\section{Proof of Theorem \ref{theorem:result}}\label{section:proof}

We fix $y,z\in\mathbb{C}^*$ such that $|q|<|y|<1$ and $|q|<|z|<1$.  We recall the function
\begin{align}
H(x,y,z;q):=F(x,y,z;q)-G(x,y,z;q), \label{equation:diff-fun}
\end{align}
where $F$  and $G$ are the respective left and right-hand sides of (\ref{equation:thm-result}).  By Proposition \ref{proposition:H-analytic}, the difference function (\ref{equation:diff-fun}) is analytic for $x\ne 0$, thus our function $H$ can be written as a Laurent series in $x$ valid for all $x\ne 0$
\begin{align}
H(x,y,z;q)=\sum_{m\in\mathbb{Z}}C_mx^m, \label{equation:H-Laurent}
\end{align}
where the $C_m$ depend on $y$, $z$, and $q$.  Proposition \ref{proposition:prop-functional} yields the functional equation
\begin{align}
H(q^2x,y,z;q)=\frac{xq}{yz}H(x,y,z;q).\label{equation:H-functional}
\end{align}
Inserting (\ref{equation:H-functional}) into the Laurent series (\ref{equation:H-Laurent}) yields 
\begin{align}
\sum_{m\in\mathbb{Z}}C_m(q^2x)^m=\frac{xq}{yz}\sum_{m\in\mathbb{Z}}C_mx^m,
\end{align}
which gives
\begin{align}
C_mq^{2m}=\frac{q}{yz}C_{m-1},
\end{align}
or
\begin{align}
C_{m+1}=\frac{q^{-1-2m}}{yz}C_{m}.
\end{align}
Iteration yields for $k\in\mathbb{Z}$
\begin{align}
C_{k}=y^{-k}z^{-k}q^{-k^2}C_0.
\end{align}
Hence
\begin{equation}
H(x,y,z;q)=C_0\sum_{k\in\mathbb{Z}}y^{-k}z^{-k}q^{-k^2}x^k.
\end{equation}
Because $H$ is analytic for $x\ne 0$, we can use, say, the ratio test to conclude that $C_0=0$.  It follow that for $x, y,z\in\mathbb{C}^*$ such that $|q|<|y|<1$, $|q|<|z|<1$, and $x$ neither zero or an integral power of $q$, we have Theorem \ref{theorem:result}:
\begin{align}
F(x,y,z;q)=G(x,y,z;q).
\end{align}

\section{Proof of Theorem \ref{theorem:same-parity}}\label{section:same}
In this section, we redefine $F(x,y,z;q)$, $G(x,y,z;q)$, $G_1(x,y,z;q)$, $G_2(x,y,z;q)$, and $H(x,y,z;q)$ for the purpose of proving Theorem \ref{theorem:same-parity}.  We define
\begin{align}
F(x,y,z;q)&:=\sum_{\sg(s)=\sg(t)}\sg(s)\frac{q^{4st}y^{2s}z^{2t}}{1-x^2q^{4s+4t}}
+x\sum_{\sg(s)=\sg(t)}\sg(s)\frac{q^{4st+4s+4t+3}y^{2s+1}z^{2t+1}}{1-x^2q^{4s+4t+4}}
\end{align}
and
\begin{align}
G(x,y,z;q)&:=\frac{J_4^3j(y^2z^2;q^4)}{j(y^2;q^4)j(z^2;q^4)} m \big (-\frac{qx}{yz},q^{2},-yz \big ) 
+\frac{J_4^3j(x^2z^2;q^4)}{j(x^2;q^4)j(z^2;q^4)} m \big (-\frac{qy}{xz},q^{2},-xz \big ) \notag\\
&\ \ \ \ \ \ \ \ \ \ \ \ \ \ \ +\frac{J_1^3J_2^3}{j(x;q)j(y;q)j(z;q)} \frac{j(xy;q^2)j(xz;q^2)j(yz;q^2)}{j(-x;q^2)j(-y;q^2)j(-z;q^2)}\\
&\ \ \ \ \ \ \ \ \ \ \ \ \ \ \  + \frac{J_4^3j(x^2y^2;q^4)}{j(x^2;q^4)j(y^2;q^4)}  m \big (-\frac{qz}{xy},q^{2},-xy \big )\notag 
\end{align}
and prove a stronger theorem, which gives as a corollary Theorem \ref{theorem:same-parity}.
\begin{theorem} For  $x,y,z\in \mathbb{C}^*$ where $|q|<|y|<1$,  $|q|<|z|<1$, and $x$ neither zero or of the form $x=\pm q^{2n}$, where $n\in\mathbb{Z}$, we have
\begin{equation}
F(x,y,z;q)=G(x,y,z;q).
\end{equation}
\end{theorem}
In particular, if we impose the additional restriction $|q|<|x|<1$ and use the geometric series, we see
{\allowdisplaybreaks \begin{align}
F(x,y,z;q)&=\sum_{\sg(s)=\sg(t)}\sg(s)\frac{q^{4st}y^{2s}z^{2t}}{1-x^2q^{4s+4t}}
+x\sum_{\sg(s)=\sg(t)}\sg(s)\frac{q^{4st+4s+4t+3}y^{2s+1}z^{2t+1}}{1-x^2q^{4s+4t+4}}\label{equation:split-F} \\
&=\sum_{\sg(r)=\sg(s)=\sg(t)}q^{4(st+rs+rt)}x^{2r}y^{2s}z^{2t}\notag \\
&\ \ \ \ \ +xyzq^3\sum_{\sg(r)=\sg(s)=\sg(t)}q^{4(st+rs+rt)}(q^4x^2)^{r}(q^4y^2)^{s}(q^4z^2)^{t}\notag \\
&=\sum_{\sg(r)=\sg(s)=\sg(t)}q^{(2r)(2s)+(2r)(2t)+(2s)(2t)}x^{2r}y^{2s}z^{2t}\notag \\
&\ \ \ \ \ +\sum_{\sg(r)=\sg(s)=\sg(t)}q^{(2r+1)(2s+1)+(2r+1)(2t+1)+(2s+1)(2t+1)}x^{2r+1}y^{2s+1}z^{2t+1}\notag\\
&=\sum_{\substack{\sg(r)=\sg(s)=\sg(t)\\r\equiv s \pmod 2 }}q^{rs+rt+st}x^{r}y^{s}z^{t}.\notag
\end{align}}%

For the remainder of this section we will give the analogs of the proposition and lemmas to what one finds in Sections \ref{section:functional} and \ref{section:analytic}.  Once that is done, the proof of Theorem \ref{theorem:same-parity} is exactly the same as what one finds in Section \ref{section:proof}, so we will omit it.    

\begin{proposition}  \label{proposition:prop-functional-same} For  $x,y,z\in \mathbb{C}^*$ where $|q|<|y|<1$, $|q|<|z|<1$, and $x$ is generic, the functions $F(x,y,z;q)$ and $G(x,y,z;q)$ satisfy the functional equation
\begin{align}
M(q^2x,y,z;q)&=\frac{xq}{yz}M(x,y,z;q) +\frac{J_4^3j(y^2z^2;q^4)}{j(y^2;q^4)j(z^2;q^4)}\\
&\ \ \ \ \ - \frac{qx}{yz} \frac{J_4^3j(x^2y^2;q^4)}{j(x^2;q^4)j(y^2;q^4)} 
  -\frac{qx}{yz} \frac{J_4^3j(x^2z^2;q^4)}{j(x^2;q^4)j(z^2;q^4)}.\notag
\end{align}
\end{proposition}
\begin{proof} 
For the function $F(x,y,z;q)$, we have
{\allowdisplaybreaks \begin{align*}
F&(q^2x,y,z;q)-\frac{xq}{yz}F(x,y,z;q)\\
&=\sum_{\sg(s)=\sg(t)}\sg(s)\frac{q^{4st}y^{2s}z^{2t}}{1-x^2q^{4s+4t+4}}
+q^2x\sum_{\sg(s)=\sg(t)}\sg(s)\frac{q^{4st+4s+4t+3}y^{2s+1}z^{2t+1}}{1-x^2q^{4s+4t+8}}\\
&\ \ \ \ \ -\frac{xq}{yz}\sum_{\sg(s)=\sg(t)}\sg(s)\frac{q^{4st}y^{2s}z^{2t}}{1-x^2q^{4s+4t}}
-\frac{x^2q}{yz}\sum_{\sg(s)=\sg(t)}\sg(s)\frac{q^{4st+4s+4t+3}y^{2s+1}z^{2t+1}}{1-x^2q^{4s+4t+4}}\\
&=\sum_{\sg(s)=\sg(t)}\sg(s)\frac{q^{4st}y^{2s}z^{2t}}{1-x^2q^{4s+4t+4}}
-\sum_{\sg(s)=\sg(t)}\sg(s)\frac{q^{4st}y^{2s}z^{2t}}{1-x^2q^{4s+4t+4}}\cdot x^2q^{4s+4t+4}\\
&\ \ \ \ \ +\frac{xq}{yz}\sum_{\sg(s)=\sg(t)}\sg(s)\frac{q^{4(s+1)(t+1)}y^{2(s+1)}z^{2(t+1)}}{1-x^2q^{4(s+1)+4(t+1)}}
 -\frac{xq}{yz}\sum_{\sg(s)=\sg(t)}\sg(s)\frac{q^{4st}y^{2s}z^{2t}}{1-x^2q^{4s+4t}}\\
 &=\sum_{\sg(s)=\sg(t)}\sg(s)q^{4st}y^{2s}z^{2t}
 -\frac{xq}{yz}\sum_{s\in\mathbb{Z}}\frac{y^{2s}}{1-x^2q^{4s}} -\frac{xq}{yz}\sum_{t\in\mathbb{Z}}\frac{z^{2t}}{1-x^2q^{4t}}\\
&=\frac{J_4^3j(y^2z^2;q^4)}{j(y^2;q^4)j(z^2;q^4)}- \frac{qx}{yz} \frac{J_4^3j(x^2y^2;q^4)}{j(x^2;q^4)j(y^2;q^4)} 
  -\frac{qx}{yz} \frac{J_4^3j(x^2z^2;q^4)}{j(x^2;q^4)j(z^2;q^4)},
\end{align*}}%
where in the last two lines we have used (\ref{equation:H2id1.15}) and (\ref{equation:kronecker-original}).  For $G(x,y,z;q),$
{\allowdisplaybreaks \begin{align*}
G&(q^2x,y,z;q)\\
&= \frac{J_4^3j(y^2z^2;q^4)}{j(y^2;q^4)j(z^2;q^4)}  m \big (-\frac{q^3x}{yz},q^{2},-yz \big ) 
+\frac{J_4^3j(q^4x^2z^2;q^4)}{j(q^4x^2;q^4)j(z^2;q^4)}  m \big (-\frac{qy}{q^2xz},q^{2},-q^2xz \big )\\
&\ \ \ \ \ +\frac{J_1^3J_2^3}{j(q^2x;q)j(y;q)j(z;q)} \frac{j(q^2xy;q^2)j(q^2xz;q^2)j(yz;q^2)}{j(-q^2x;q^2)j(-y;q^2)j(-z;q^2)}\\
&\ \ \ \ \  + \frac{J_4^3j(q^4x^2y^2;q^4)}{j(q^4x^2;q^4)j(y^2;q^4)}  m \big (-\frac{qz}{q^2xy},q^{2},-q^2xy \big ) \\
&=\frac{qx}{yz} G(x,y,z;q)+\frac{J_4^3j(y^2z^2;q^4)}{j(y^2;q^4)j(z^2;q^4)}  -\frac{qx}{yz} \frac{J_4^3j(x^2z^2;q^4)}{j(x^2;q^4)j(z^2;q^4)}- \frac{qx}{yz} \frac{J_4^3j(x^2y^2;q^4)}{j(x^2;q^4)j(y^2;q^4)},
\end{align*}}%
where we have used (\ref{equation:mxqz-fnq-z}), (\ref{equation:mxqz-fnq-x}), and (\ref{equation:j-elliptic}).
\end{proof}

We decompose the function
\begin{equation}
G(x,y,z;q)=G_1(x,y,z;q)+G_2(x,y,z;q),
\end{equation}
where
\begin{align}
G_1&(x,y,z;q):=\frac{j(yz;q^2)}{j(y^2;q^4)j(z^2;q^4)}
  \frac{J_4^4}{J_2^2}\sum_{k\in\mathbb{Z}}\frac{q^{k^2-k}y^kz^k}{1-q^{2k-1}x}\label{equation:G1-same-def}\\
&\ \ \ \ \    +\frac{j(xz;q^2)}{j(x^2;q^4)j(z^2;q^4)}   \frac{J_4^4}{J_2^2}\sum_{k\in\mathbb{Z}}\frac{q^{k^2-k}x^kz^k}{1-q^{2k-1}y}   + \frac{j(xy;q^2)}{j(x^2;q^4)j(y^2;q^4)} \frac{J_4^4}{J_2^2}\sum_{k\in\mathbb{Z}}\frac{q^{k^2-k}x^ky^k}{1-q^{2k-1}z}  \notag
\end{align}
and
\begin{align}
G_2(x,y,z;q)=\frac{J_1^3J_2^3}{j(x;q)j(y;q)j(z;q)} \frac{j(xy;q^2)j(xz;q^2)j(yz;q^2)}{j(-x;q^2)j(-y;q^2)j(-z;q^2)}.
\end{align}

For $x\ne 0$, potential singularities of the function
\begin{align}
H(x,y,z;q):=F(x,y,z;q)-G(x,y,z;q),
\end{align}
are limited to simple poles at $x=q^n$ and $x=-q^{2n}$ for $n\in \mathbb{Z}$.  The following series of lemmas demonstrate that the respective residues of any such poles always sum to zero.  Hence, $H(x,y,z;q)$ is analytic for $x\ne 0$, and one then proceeds as in Section \ref{section:proof}.

\begin{lemma} \label{lemma:F1-n-pos-same} For fixed $y,z\in\mathbb{C}^*$ where $|q|<|y|<1$ and $|q|<|z|<1$, the function $F(x,y,z;q)$ has simple poles at $x_0^2=q^{4n}$, where $n\in\mathbb{Z}$, with respective residues
\begin{equation}
\frac{1}{2} \sum_{s=1}^{2n-1}\frac{q^{s(2n-s)+2n}}{y^{s}z^{2n-s}}
\end{equation}
for $x_0=q^{2n}$, and 
\begin{equation}
-\frac{1}{2} \sum_{s=1}^{2n-1}\frac{(-1)^sq^{s(2n-s)+2n}}{y^{s}z^{2n-s}}
\end{equation}
for $x_0=-q^{2n}.$
\end{lemma}
\begin{proof}  Without loss of generality, we assume $n\ge 1$.  We have
{\allowdisplaybreaks \begin{align*}
&\lim_{x\rightarrow x_0}(x-x_0)F(x,y,z;q)\\
&=\lim_{x\rightarrow x_0}(x-x_0)\Big [ \sum_{\sg(s)=\sg(t)}\sg(s)\frac{q^{4st}y^{2s}z^{2t}}{1-x^2q^{4s+4t}}
+x\sum_{\sg(s)=\sg(t)}\sg(s)\frac{q^{4st+4s+4t+3}y^{2s+1}z^{2t+1}}{1-x^2q^{4s+4t+4}}\Big]\\
&=-\lim_{x\rightarrow x_0}(x-x_0)\Big [ \sum_{\substack{s+t=-n\\s,t<0}}\frac{q^{4st}y^{2s}z^{2t}}{1-x^2q^{4s+4t}}
+x\sum_{\substack{s+t+1=-n\\s,t<0}}\frac{q^{4st+4s+4t+3}y^{2s+1}z^{2t+1}}{1-x^2q^{4s+4t+4}}\Big]\\
&=-\lim_{x\rightarrow x_0}(x-x_0)\Big [ \sum_{\substack{s+t=n\\s,t\ge1}}\frac{q^{4st}y^{-2s}z^{-2t}}{1-x^2q^{-4n}}
 +x\sum_{\substack{s+t-1=n\\s,t\ge 1}}\frac{q^{4st-4s-4t+3}y^{-2s+1}z^{-2t+1}}{1-x^2q^{-4n}}\Big]\\
 &=\lim_{x\rightarrow x_0}(x-x_0)\Big [ \sum_{s=1}^{n-1}\frac{q^{4sn-4s^2}y^{-2s}z^{-(2n-2s)}q^{4n}}{(x-q^{2n})(x+q^{2n})}\\
&\ \ \ \ \ \ \ \ \ \ \ \ \ \ \   +x\sum_{s=1}^{n}\frac{q^{4sn-4s^2+4s-4n-1}y^{-2s+1}z^{-2n-2s-1}q^{4n}}{(x-q^{2n})(x+q^{2n})}\Big].
 \end{align*}}%
 For $x_0=q^{2n}$,
 {\allowdisplaybreaks \begin{align*}
\lim_{x\rightarrow x_0}(x-x_0)F(x,y,z;q)
  &=\frac{1}{2} \sum_{s=1}^{n-1}\frac{q^{4sn-4s^2+2n}}{y^{2s}z^{2n-2s}}
 +\frac{1}{2}\sum_{s=1}^{n}\frac{q^{4sn-4s^2+4s-1}}{y^{2s-1}z^{2n-2s+1}}\\
  &=\frac{1}{2} \sum_{s=1}^{n-1}\frac{q^{2s(2n-2s)+2n}}{y^{2s}z^{2n-2s}}
 +\frac{1}{2}\sum_{s=1}^{n}\frac{q^{(2s-1)2n-(2s-1)^2+2n}}{y^{2s-1}z^{2n-2s+1}}\\
&=\frac{1}{2} \sum_{s=1}^{2n-1}\frac{q^{s(2n-s)+2n}}{y^{s}z^{2n-s}}.
\end{align*}}%
The case $x_0=-q^{2n}$ is similar, so we omit it.
\end{proof}

\begin{lemma} \label{lemma:G1-n-odd-same} For fixed $y,z\in\mathbb{C}^*$ where $|q|<|y|<1$ and $|q|<|z|<1$, the function $G_1(x,y,z;q)$ has simple poles at $x_0=q^{2n+1}$, where $n\in\mathbb{Z}$, with respective residues
\begin{equation}
-\frac{j(yz;q^2)}{j(y^2;q^4)j(z^2;q^4)}\cdot \frac{J_4^4}{J_2^2}\cdot \frac{q^{n^2+3n+1}}{y^{n}z^{n}}.
\end{equation}
\end{lemma}
\begin{proof}  This follows immediately from definition (\ref{equation:G1-same-def}).
\end{proof}

\begin{lemma} \label{lemma:G1-n-sq-same} For fixed $y,z\in\mathbb{C}^*$ where $|q|<|y|<1$ and $|q|<|z|<1$, the function $G_1(x,y,z;q)$ has simple poles at $x_0^2=q^{4n}$, where $n\in\mathbb{Z}$, with respective residues
\begin{equation}
\frac{1}{2}\sum_{s=1}^{2n-1}\frac{q^{s(2n-s)+2n}}{y^{s}z^{2n-s}}
+\frac{q^{n^2+2n}}{2y^nz^n}\cdot \frac{J_2^8}{J_1^2J_4^2}\cdot \frac{j(yz;q^2)}{j(qy;q^2)j(-y;q^2)j(-z;q^2)j(qz;q^2)}
\end{equation}
for $x_0=q^{2n}$, and 
\begin{equation}
-\frac{1}{2}\sum_{s=1}^{2n-1}\frac{(-1)^sq^{s(2n-s)+2n}}{y^{s}z^{2n-s}}
 - \frac{(-1)^nq^{n^2+2n}}{2z^{n}y^{n}}\cdot \frac{J_1^4}{J_2^2}\cdot \frac{j(yz;q^2)}{j(y;q)j(z;q)}
\end{equation}
for $x_0=-q^{2n}$.
\end{lemma}
\begin{proof}   Beginning with Proposition \ref{proposition:H1Thm1.3} where $x_0=q^{2n}$
{\allowdisplaybreaks \begin{align*}
\lim_{x\rightarrow x_0}&(x-x_0)G_1(x,y,z;q)\\
&=\lim_{x\rightarrow x_0}(x-x_0)\frac{1}{j(x^2;q^4)}\Big [ \frac{j(xz;q^2)}{j(z^2;q^4)}   \frac{J_4^4}{J_2^2}\sum_{k\in\mathbb{Z}}\frac{q^{k^2-k}x^kz^k}{1-q^{2k-1}y} 
   + \frac{j(xy;q^2)}{j(y^2;q^4)} \frac{J_4^4}{J_2^2}\sum_{k\in\mathbb{Z}}\frac{q^{k^2-k}x^ky^k}{1-q^{2k-1}z}  \Big ]\\
&=\frac{(-1)^{n+1}q^{4\binom{n}{2}}q^{2n}}{2J_{4}^3} \frac{J_4^4}{J_2^2}\Big [ \frac{j(q^{2n}z;q^2)}{j(z^2;q^4)}   \sum_{k\in\mathbb{Z}}\frac{q^{k^2-k}(q^{2n}z)^k}{1-q^{2k-1}y} 
   + \frac{j(q^{2n}y;q^2)}{j(y^2;q^4)} \sum_{k\in\mathbb{Z}}\frac{q^{k^2-k}(q^{2n}y)^k}{1-q^{2k-1}z}  \Big ]\\
&=\frac{(-1)^{n+1}q^{4\binom{n}{2}}q^{2n}}{2J_{4}^3} \frac{J_4^4}{J_2^2}\Big [ \frac{j(q^{2n}z;q^2)}{j(z^2;q^4)}   
 j(-q^{2n}z;q^2) m\big (-\frac{qy}{q^{2n}z} ,q^{2},-q^{2n}z\big ) \\
&\ \ \ \ \    + \frac{j(q^{2n}y;q^2)}{j(y^2;q^4)}  j(-q^{2n}y;q^2) m\big (-\frac{qz}{q^{2n}y} ,q^{2},-q^{2n}y\big )   \Big ]\\
&=-\frac{q^{2n}}{2} 
\Big [ \frac{1}{z^{2n}} m\big (-\frac{qy}{q^{2n}z} ,q^{2},-z\big )    + \frac{1}{y^{2n}}  m\big (-\frac{qz}{q^{2n}y} ,q^{2},-y\big )   \Big ],
\end{align*}}%
where the last two equalities follow from (\ref{equation:mdef-eq}), (\ref{equation:j-elliptic}), and (\ref{equation:j-mod-dec}).  We note
 {\allowdisplaybreaks \begin{align*}
&\lim_{x\rightarrow x_0}(x-x_0)G_1(x,y,z;q)\\
&= \frac{q^{4n}z}{2qyz^{2n}} m\big (-\frac{q^{2n}z}{qy} ,q^{2},-\frac{1}{z}\big )   
 +  \frac{q^{4n}y}{2qzy^{2n}}  m\big (-\frac{q^{2n}y}{qz} ,q^{2},-\frac{1}{y}\big ) &\textup{(by (\ref{equation:mxqz-flip})})\\
 &= \frac{q^{4n}z}{2qyz^{2n}} 
 \Big ( \sum_{k=0}^{n-1}q^{2k(n-1)-2\binom{k}{2}}(\frac{z}{qy})^k+q^{2\binom{n}{2}}(\frac{z}{qy})^n m\big (-\frac{z}{qy} ,q^{2},-\frac{1}{z}\big ) \Big) &\textup{(by (\ref{equation:mxqz-induction})}) \\
&\ \ \ \ \  +  \frac{q^{4n}y}{2qzy^{2n}}
 \Big ( \sum_{k=0}^{n-1}q^{2k(n-1)-2\binom{k}{2}}(\frac{y}{qz})^k+q^{2\binom{n}{2}}(\frac{y}{qz})^n m\big (-\frac{y}{qz} ,q^{2},-\frac{1}{y}\big ) \Big)   \\
 &= \sum_{k=0}^{n-1}\frac{q^{2kn-(k+1)^2+4n}}{2y^{k+1}z^{2n-k-1}}+\frac{z}{qy}\frac{q^{n^2+2n}}{2y^{n}z^{n}} m\big (-\frac{z}{qy} ,q^{2},-\frac{1}{z}\big )  \\
&\ \ \ \ \  +  \sum_{k=0}^{n-1}\frac{q^{2kn-(k+1)^2+4n}}{2z^{k+1}y^{2n-k-1}}
+\frac{y}{zq}\frac{q^{n^2+2n}}{2z^{n}y^{n}}m\big (-\frac{y}{qz} ,q^{2},-\frac{1}{y}\big ) \\
&= \sum_{k=1}^{n}\frac{q^{2kn-k^2+2n}}{2y^{k}z^{2n-k}}+\frac{z}{qy}\frac{q^{n^2+2n}}{2y^{n}z^{n}} m\big (-\frac{z}{qy} ,q^{2},qy\big ) 
&\textup{(by (\ref{equation:mxqz-fnq-newz}))} \\
&\ \ \ \ \  +  \sum_{k=1}^{n}\frac{q^{2kn-k^2+2n}}{2z^{k}y^{2n-k}}
-\frac{q^{n^2+2n}}{2z^{n}y^{n}}m\big (-\frac{qz}{y} ,q^{2},- y\big ) &\textup{(by (\ref{equation:mxqz-flip})}) \\
&= \sum_{k=1}^{n}\frac{q^{2kn-k^2+2n}}{2y^{k}z^{2n-k}}+\frac{z}{qy}\frac{q^{n^2+2n}}{2y^{n}z^{n}} m\big (-\frac{z}{qy} ,q^{2},qy\big )  \\
&\ \ \ \ \  +  \sum_{k=1}^{n}\frac{q^{2kn-k^2+2n}}{2z^{k}y^{2n-k}}
-\frac{q^{n^2+2n}}{2z^{n}y^{n}}-\frac{q^{n^2+2n}}{2z^{n}y^{n}}\frac{z}{qy}m\big (-\frac{qz}{y} ,q^{2},- y\big )
&\textup{(by (\ref{equation:mxqz-fnq-x}))} \\
&=\frac{1}{2}\sum_{s=1}^{2n-1}\frac{q^{s(2n-s)+2n}}{y^{2}z^{2n-s}}
+\frac{z}{qy}\frac{q^{n^2+2n}}{2y^{n}z^{n}} \Big [ m\big (-\frac{z}{qy} ,q^{2},qy\big ) -m\big (-\frac{qz}{y} ,q^{2},- y\big ) \Big],
\end{align*}}%
and the result follows from (\ref{equation:change-n-even}) and the product rearrangement $\overline{J}_{1,2}=J_2^5/J_1^2J_4^2$.  The argument for the case $x_0=-q^{2n}$ is similar, so we omit it.
\end{proof}

\begin{lemma} \label{lemma:G2-n-neg-same} For fixed $y,z\in\mathbb{C}^*$ where $|q|<|y|<1$ and $|q|<|z|<1$, the function $G_2(x,y,z;q)$ has simple poles at $x_0=-q^{2n}$, where $n\in\mathbb{Z}$, with respective residues
\begin{equation}
\frac{(-1)^nq^{n^2+2n}}{y^nz^n}\cdot \frac{J_1^4}{J_2^2}\cdot \frac{j(yz;q^2)}{j(y;q)j(z;q)}.
\end{equation}
\end{lemma}
\begin{proof}  Beginning with Proposition \ref{proposition:H1Thm1.3},
{\allowdisplaybreaks \begin{align*}
\lim_{x\rightarrow x_0}&(x-x_0)G_2(x,y,z;q)\\
&=\lim_{x\rightarrow x_0}(x-x_0)\frac{J_1^3J_2^3}{j(x;q)j(y;q)j(z;q)} \frac{j(xy;q^2)j(xz;q^2)j(yz;q^2)}{j(-x;q^2)j(-y;q^2)j(-z;q^2)}\\
&=-\frac{(-1)^{n+1}q^{2\binom{n}{2}}q^{2n}}{J_2^3} \frac{J_1^3J_2^3}{j(-q^{2n};q)j(y;q)j(z;q)} \frac{j(-q^{2n}y;q^2)j(-q^{2n}z;q^2)j(yz;q^2)}{j(-y;q^2)j(-z;q^2)}\\
&=\frac{(-1)^nq^{n^2+2n}}{y^nz^n} \frac{J_1^4}{2J_2^2} \frac{j(yz;q^2)}{j(y;q)j(z;q)}, 
\end{align*}}%
where we have used (\ref{equation:j-elliptic}), (\ref{equation:j-mod-2}), and the fact that $j(-1;q)=2J_2^2/J_1$.
\end{proof}

\begin{lemma} \label{lemma:G2-n-same} For fixed $y,z\in\mathbb{C}^*$ where $|q|<|y|<1$ and $|q|<|z|<1$, the function $G_2(x,y,z;q)$ has simple poles at $x_0=q^{n}$, where $n\in\mathbb{Z}$, with respective residues
\begin{equation}
- \frac{q^{n^2+2n}}{2y^nz^n}\cdot\frac{J_2^8}{J_1^2J_4^2}\cdot
 \frac{j(yz;q^2)}{j(qy;q^2)j(qz;q^2)j(-y;q^2)j(-z;q^2)}
\end{equation}
for $x_0=q^{2n}$, and 
\begin{equation}
\frac{j(yz;q^2)}{j(y^2;q^4)j(z^2;q^4)} \cdot \frac{J_4^4}{J_2^2}\cdot \frac{q^{n^2+3n+1}}{y^nz^n}
\end{equation}
for $x_0=q^{2n+1}$.
\end{lemma}
\begin{proof}  Beginning with Proposition \ref{proposition:H1Thm1.3}, where $x_0=q^{2n}$
{\allowdisplaybreaks \begin{align*}
&\lim_{x\rightarrow x_0}(x-x_0)G_2(x,y,z;q)\\
&=\lim_{x\rightarrow x_0}(x-x_0)\frac{J_2^2}{J_1} \frac{J_1^3J_2^3}{j(x;q^2)j(xq;q^2)j(y;q)j(z;q)} \frac{j(xy;q^2)j(xz;q^2)j(yz;q^2)}{j(-x;q^2)j(-y;q^2)j(-z;q^2)}&(\textup{by }(\ref{equation:j-mod-2}))\\
&=\frac{J_2^2}{J_1} \frac{(-1)^{n+1}q^{2\binom{n}{2}}q^{2n}}{J_2^3}\frac{J_1^3J_2^3}{j(q^{2n+1};q^2)j(y;q)j(z;q)} \frac{j(q^{2n}y;q^2)j(q^{2n}z;q^2)j(yz;q^2)}{j(-q^{2n};q^2)j(-y;q^2)j(-z;q^2)}\\
&= -  \frac{q^{n^2+2n}}{y^nz^n}
 \frac{J_1^2J_2^2}{j(q;q^2)j(y;q)j(z;q)} \frac{j(y;q^2)j(z;q^2)j(yz;q^2)}{j(-1;q^2)j(-y;q^2)j(-z;q^2)}\\
 &=  -  \frac{q^{n^2+2n}}{y^nz^n}\frac{J_2^6}{j(q;q^2)j(-1;q^2)}
 \frac{j(yz;q^2)}{j(qy;q^2)j(qz;q^2)j(-y;q^2)j(-z;q^2)},
\end{align*}}%
where in the last two equalities we have used  (\ref{equation:j-elliptic}) and (\ref{equation:j-mod-2}).  The result then follows from (\ref{equation:j-mod-2}) and the product rearrangements $J_{1,2}=J_1^2/J_2$, $\overline{J}_{0,2}=2J_4^2/J_2$, and $\overline{J}_{1,2}=J_2^5/J_1^2J_4^2$. The argument for the case $x_0=q^{2n+1}$ is similar, so we omit it.
\end{proof}

\section{Proof of Theorem \ref{theorem:different-parity}}\label{section:different}

We again redefine $F(x,y,z;q)$, $G(x,y,z;q)$, $G_1(x,y,z;q)$, $G_2(x,y,z;q)$, and $H(x,y,z;q)$, but this time for the purpose of proving Theorem \ref{theorem:different-parity}.  We define
\begin{align}
F(x,y,z;q)&:=z\sum_{\sg(s)=\sg(t)}\sg(s)\frac{q^{4st+2s}y^{2s}z^{2t}}{1-x^2q^{4s+4t+2}}
+xyq\sum_{\sg(s)=\sg(t)}\sg(s)\frac{q^{4st+2s+4t}y^{2s}z^{2t}}{1-x^2q^{4s+4t+2}}
\end{align}
and
\begin{align}
G&(x,y,q)\\
&: =  \frac{zJ_4^3j(q^2y^2z^2;q^4)}{j(q^2y^2;q^4)j(z^2;q^4)}  m \big (-\frac{qx}{yz},q^{2},-qyz \big ) 
+  \frac{zJ_4^3j(q^2x^2z^2;q^4)}{j(q^2x^2;q^4)j(z^2;q^4)} m \big (-\frac{qy}{xz},q^{2}, -qxz\big ) \notag \\
&\ \ \ \ \ - \frac{zJ_1^3J_2^3}{j(x;q)j(y;q)j(z;q)}\frac{j(xy;q^2)j(qxz;q^2)j(qyz;q^2)}{j(-xq;q^2)j(-yq;q^2)j(-z;q^2)}\notag \\
&\ \ \ \ \   - \frac{q}{xy} \frac{J_4^3j(x^2y^2;q^4)}{j(q^2x^2;q^4)j(q^2y^2;q^4)}  m \big (-\frac{qz}{xy},q^{2},-xy \big ).\notag 
\end{align}
and prove a stronger theorem, which gives as a corollary Theorem \ref{theorem:different-parity}.
\begin{theorem} For  $x,y,z\in \mathbb{C}^*$ where $|q|<|y|<1$,  $|q|<|z|<1$, and $x$ neither zero or of the form $x=\pm q^{2n+1}$, where $n\in\mathbb{Z}$, we have
\begin{equation}
F(x,y,z;q)=G(x,y,z;q).
\end{equation}
\end{theorem}
In particular, if we further restrict $q<|x|<1$, then the geometric series yields
\begin{align}
F(x,y,z;q)=\sum_{\substack{\sg(r)=\sg(s)=\sg(t)\\r\equiv s\not \equiv t \pmod 2}}&q^{rs+rt+st}x^ry^sz^t.
\end{align}

For the remainder of this section we will give the analogs of the proposition and lemmas to what one finds in Sections \ref{section:functional} and \ref{section:analytic}.  Once that is done, the proof is exactly the same as what one finds in Section \ref{section:proof}, so we will omit it.  We also omit the proofs to the proposition and lemmas because they are similar to those of the previous sections.

\begin{proposition}  \label{proposition:prop-functional-diff} For  $x,y,z\in \mathbb{C}^*$ where $|q|<|y|<1$, $|q|<|z|<1$, and $x$ is generic, the functions $F(x,y,z;q)$ and $G(x,y,z;q)$ satisfy the functional equation
\begin{align}
M(q^2x,y,z;q)&=\frac{xq}{yz}M(x,y,z;q)+z \cdot \frac{J_4^3j(q^2y^2z^2;q^4)}{j(q^2y^2;q^4)j(z^2;q^4)}\\
&\ \ \ \ \ -\frac{xq}{y}\cdot \frac{J_4^3j(q^2x^2z^2;q^4)}{j(q^2x^2;q^4)j(z^2;q^4)}+\frac{q^2}{zy^2}\cdot \frac{J_4^3j(x^2y^2;q^4)}{j(q^2x^2;q^4)j(q^2y^2;q^4)}.\notag
\end{align}
\end{proposition}

We decompose the function
\begin{equation}
G(x,y,z;q)=G_1(x,y,z;q)+G_2(x,y,z;q),
\end{equation}
where
\begin{align}
G_1&(x,y,q): =z  \frac{j(qyz;q^2)}{j(q^2y^2;q^4)j(z^2;q^4)} 
 \frac{J_4^4}{J_2^2}\sum_{k}\frac{q^{k^2}(yz)^k}{1-q^{2k}x} \\
&  + z  \frac{j(qxz;q^2)}{j(q^2x^2;q^4)j(z^2;q^4)} 
 \frac{J_4^4}{J_2^2}\sum_{k}\frac{q^{k^2}(xz)^k}{1-q^{2k}y}
  - \frac{q}{xy}  \frac{j(xy;q^2)}{j(q^2x^2;q^4)j(q^2y^2;q^4)} 
 \frac{J_4^4}{J_2^2}\sum_{k}\frac{q^{k^2-k}(xy)^k}{1-q^{2k-1}z}\notag 
\end{align}
and
\begin{align}
G_2(x,y,q): = - \frac{zJ_1^3J_2^3}{j(x;q)j(y;q)j(z;q)}\frac{j(xy;q^2)j(qxz;q^2)j(qyz;q^2)}{j(-xq;q^2)j(-yq;q^2)j(-z;q^2)}.
\end{align}

For $x\ne 0$, potential singularities of the function
\begin{align}
H(x,y,z;q):=F(x,y,z;q)-G(x,y,z;q),
\end{align}
are limited to simple poles at $x=q^n$ and $x=-q^{2n+1}$ for $n\in \mathbb{Z}$.  The following series of lemmas demonstrate that the respective residues of any such poles always sum to zero.  Hence, $H(x,y,z;q)$ is analytic for $x\ne 0$, and one then proceeds as in Section \ref{section:proof}.

\begin{lemma} \label{lemma:F-n-diff} For fixed $y,z\in\mathbb{C}^*$ where $|q|<|y|<1$ and $|q|<|z|<1$, the function $F(x,y,z;q)$ has simple poles at $x_0^2=q^{4n+2}$, where $n\in\mathbb{Z}$, with respective residues
\begin{equation}
\frac{1}{2} \cdot \sum_{s=1}^{2n}\frac{q^{s(2n-s+1)+2n+1}}{y^{s}z^{2n-s+1}}
\end{equation}
for $x_0=q^{2n+1}$, and
\begin{equation}
-\frac{1}{2} \cdot \sum_{s=1}^{2n}\frac{(-1)^sq^{s(2n-s+1)+2n+1}}{y^{s}z^{2n-s+1}}
\end{equation}
for $x_0=-q^{2n+1}$.
\end{lemma}

\begin{lemma} \label{lemma:G1-sq-diff} For fixed $y,z\in\mathbb{C}^*$ where $|q|<|y|<1$ and $|q|<|z|<1$, the function $G_1(x,y,z;q)$ has simple poles at $x_0^2=q^{4n+2}$, where $n\in\mathbb{Z}$, with respective residues
\begin{equation}
  \frac{1}{2} \cdot\sum_{k=1}^{2n}  \frac{q^{k(2n-k+1)+2n+1}}{y^{k}z^{2n-k+1}}
  - \frac{1}{2} \cdot \frac{q^{n^2+3n+1}}{y^{n}z^{n}}
\cdot \frac{J_2^8}{J_1^2J_4^2}\cdot \frac{j(qyz;q^2)}{j(y;q^2)j(-z;q^2)j(-qy;q^2)j(qz;q^2)}
\end{equation}
for $x_0=q^{2n+1}$, and
\begin{equation}
 -\frac{1}{2}\cdot \sum_{k=1}^{2n}\frac{(-1)^kq^{k(2n-k+1)+2n+1}}{y^{k}z^{2n-k+1}}  
  -\frac{1}{2}\cdot  \frac{(-1)^nq^{n^2+3n+1}}{y^{n}z^{n}} 
\cdot \frac{J_1^4}{J_2^2}\cdot \frac{j(qyz;q^2)}{j(y;q)j(z;q)}
\end{equation}
for $x_0=-q^{2n+1}$.
\end{lemma}

\begin{lemma} \label{lemma:G1-n-diff} For fixed $y,z\in\mathbb{C}^*$ where $|q|<|y|<1$ and $|q|<|z|<1$, the function $G_1(x,y,z;q)$ has simple poles at $x_0=q^{2n}$, where $n\in\mathbb{Z}$, with respective residues
\begin{equation}
- \frac{zq^{n^2+2n}}{y^nz^n}\cdot \frac{J_4^4}{J_2^2} \cdot  \frac{j(qyz;q^2)}{j(q^2y^2;q^4)j(z^2;q^4)} .
\end{equation}
\end{lemma}

\begin{lemma} \label{lemma:G2-n-diff} For fixed $y,z\in\mathbb{C}^*$ where $|q|<|y|<1$ and $|q|<|z|<1$, the function $G_2(x,y,z;q)$ has simple poles at $x_0=q^{n}$, where $n\in\mathbb{Z}$, with respective residues
\begin{equation}
  \frac{zq^{n^2+2n}}{y^nz^n}\cdot  \frac{J_4^4}{J_2^2}\cdot   \frac{j(qyz;q^2)}{j(y^2q^2;q^4)j(z^2;q^4)}
\end{equation}
for $x_0=q^{2n}$, and 
\begin{equation}
\frac{1}{2}\cdot \frac{q^{n^2+3n+1}}{y^nz^n}\cdot  \frac{J_2^8}{J_1^2J_4^2}\cdot 
\frac{j(qyz;q^2)}{j(y;q^2)j(-z;q^2)j(-yq;q^2)j(qz;q^2)}
\end{equation}
for $x_0=q^{2n+1}$. 
\end{lemma}

\begin{lemma} \label{lemma:G2-n-neg-diff} For fixed $y,z\in\mathbb{C}^*$ where $|q|<|y|<1$ and $|q|<|z|<1$, the function $G_2(x,y,z;q)$ has simple poles at $x_0=-q^{2n+1}$, where $n\in\mathbb{Z}$, with respective residues
\begin{equation}
\frac{1}{2}\cdot \frac{(-1)^{n}q^{n^2+3n+1}}{y^nz^{n}}\cdot \frac{J_1^4}{J_2^2}\cdot \frac{j(qyz;q^2)}{j(y;q)j(z;q)}.
\end{equation}
\end{lemma}

\section{Concluding remarks}\label{section:starting}
Understanding the starting point for our initial guesses should prove useful for related and higher-dimensional generalizations.  We define
\begin{equation}
\mathcal{F}(x,y,z;q):=\Big ( \sum_{r,s,t \ge 0}+\sum_{r,s,t<0} \Big)q^{rs+rt+st}x^ry^sz^t,
\end{equation}
and state the following lemma
\begin{lemma} \label{lemma:starting_point} For $|q|<|y|<1$ and $|q|<|z|<1$, we have
\begin{align}
&\mathcal{F}(x,y,z;q)-x\mathcal{F}(x,qy,qz;q)=\frac{J_1^3j(yz;q)}{j(y;q)j(z;q)}.\label{equation:triple-functional}
\end{align}
\end{lemma}
\begin{proof} The proof is just a straightforward shift of indices
{\allowdisplaybreaks \begin{align*}
&\mathcal{F}(x,y,z;q)-x\mathcal{F}(x,qy,qz;q)\\
&=\Big ( \sum_{r,s,t \ge 0}+\sum_{r,s,t<0} \Big)q^{rs+rt+st}x^ry^sz^t-\Big ( \sum_{r,s,t \ge 0}+\sum_{r,s,t<0} \Big)q^{(r+1)s+(r+1)t+st}x^{r+1}y^sz^t\\
&=\Big ( \sum_{r,s,t \ge 0}+\sum_{r,s,t<0} \Big)q^{rs+rt+st}x^ry^sz^t\\
&\ \ \ \ \ -\Big [ \Big ( \sum_{r,s,t \ge 0}+\sum_{r,s,t<0} \Big)q^{rs+rt+st}x^{r}y^sz^t  -\sum_{s,t}\sg(s,t)q^{st}y^sz^t\Big ].
\end{align*}}%
Using (\ref{equation:kronecker}), the result follows.
\end{proof}

If we iterate (\ref{equation:triple-functional}), we obtain,
{\allowdisplaybreaks \begin{align*}
\mathcal{F}(x,y,z;q)&\sim \frac{J_1^3j(yz;q)}{j(y;q)j(z;q)}\cdot \Big [ 1+x q^{-1}(yz)^{-1} +x^2 \cdot q^{-4}(yz)^{-2} 
+x^3 q^{-9}(yz)^{-3}+\cdots\\
&\sim \frac{J_1^3j(yz;q)}{j(y;q)j(z;q)}\cdot \Big [ 1+(qx/yz) q^{-2} +(qx/yz)^2 q^{-6}
+(qx/yz)^3q^{-12}+\cdots\\
&\sim \frac{J_1^3j(yz;q)}{j(y;q)j(z;q)}\cdot \sum_{k\ge 0}(-1)^k(-qx/yz)^kq^{-2\binom{k+1}{2}}\\
&\sim \frac{J_1^3j(yz;q)}{j(y;q)j(z;q)}\cdot m (-qx/yz,q^{2},*),
\end{align*}}%
which suggests as our starting point for Theorem \ref{theorem:result}
{\allowdisplaybreaks \begin{align*}
\mathcal{F}(x,y,z;q)&\sim  \frac{J_1^3j(yz;q)}{j(y;q)j(z;q)}\cdot m (-qx/yz,q^{2},*)+ \frac{J_1^3j(xy;q)}{j(x;q)j(y;q)}\cdot m (-qz/xy,q^{2},*)\\
&\ \ \ \ \ + \frac{J_1^3j(xz;q)}{j(x;q)j(z;q)}\cdot m (-qy/xz,q^{2},*).
\end{align*}}%

One could also ask if there are proofs alternate to considering functional equations, poles, and residues.  For example, the second referee wondered if multivariate extensions of Ramanujan's ${}_1\psi_1$ notation, see for example \cite{Warn}, could be used to prove Theorem \ref{theorem:result} or similar results.

\end{document}